\documentclass[12pt, reqno]{amsart}
\usepackage{amsthm,amsmath,amssymb,url,cite,color}

\numberwithin{equation}{section}

\newtheorem{thm}{Theorem}

\newtheorem{prop}[thm]{Proposition}

\newtheorem{coro}[thm]{Corollary}
\newtheorem{Prob}[thm]{Problem}

\newtheorem{conj}[thm]{Conjecture}

\def\Z{\mathbb{Z}}
\def\N{\mathbb{N}}

\title[Two  new triangles of $q$-integers] {Two  new triangles of $q$-integers via $q$-Eulerian polynomials of type $A$ and $B$}

\author{Guoniu Han}
\address{Institut de Recherche Math\'ematique Avanc\'ee,
Universit\'e de Strasbourg et CNRS,
7, rue Ren\'e-Descartes, F-67084 Strasbourg, France.}

\author{Fr\'ed\'eric Jouhet}
\address{Universit\'e de Lyon, Universit\'e Lyon I, 
CNRS, UMR 5208 Institut Camille Jordan,
43, bd du 11 Novembre 1918, 69622 Villeurbanne Cedex, France.}

\author{Jiang Zeng}
\address{Universit\'e de Lyon, Universit\'e Lyon I, 
CNRS, UMR 5208 Institut Camille Jordan,
43, bd du 11 Novembre 1918, 69622 Villeurbanne Cedex, France.}

\dedicatory{Dedicated to Mourad Ismail and Dennis Stanton}

\date{\today}  

\begin{document}

\begin{abstract}
The classical Eulerian polynomials can be expanded in the basis 
 $t^{k-1}(1+t)^{n+1-2k}$ ($1\leq k\leq\lfloor (n+1)/2\rfloor$) with positive integral  coefficients. This formula implies 
  both the symmetry and the unimodality of the Eulerian polynomials. 
In this paper,  we prove a $q$-analogue of this expansion 
for  Carlitz's $q$-Eulerian polynomials as well as a similar formula for Chow-Gessel's $q$-Eulerian polynomials of type $B$.
We shall  give  some applications of these two formulae, which   involve two new sequences of polynomials in the variable $q$ with positive integral  coefficients.   An open problem is to give a combinatorial interpretation for these polynomials.
\end{abstract} 

\maketitle 

\section{Introduction}
The \emph{Eulerian polynomials} $A_n(t):=\sum_{k=1}^{n}A_{n,k}t^{k-1}$ (see \cite{FS70, Fo09,St97})  may be defined by
$$
\sum_{k\geq 1}k^n t^k=\frac{A_n(t)}{(1-t)^{n+1}}\qquad (n\in \N).
$$
It is well known (see \cite{FS70, SGW83}) that there are nonnegative integers $a_{n,k}$ such that 
\begin{align}\label{eq:FS}
A_{n}(t)=\sum_{k=1}^{\lfloor (n+1)/2\rfloor}a_{n,k}t^{k-1}(1+t)^{n+1-2k}.
\end{align}
For example, for $n=1,\ldots,4$, the identity reads 
 \begin{align*}
 A_1(t)=1,\quad A_2(t)=1+t,\quad 
A_3(t)=(1+t)^{2}+2 t^2,\quad
A_4(t) =(1+t)^3 + 8 t(1+t).
  \end{align*}
 In particular, this formula implies both the \emph{symmetry} and the \emph{unimodality} (see for instance \cite{Br08} for the definitions) of the Eulerian numbers $(A_{n,k})_{1\leq k\leq n}$ for any fixed $n$.  The coefficients $a_{n,k}$ 
 defined by (\ref{eq:FS}) satisfy the following recurrence relation:
\begin{equation} \label{eq:recu1}
a_{n,k}= ka_{n-1,k}+2(n+2-2k)a_{n-1,k-1} 
 \end{equation}
for $n\geq 2$ and $1\leq k \leq \lfloor (n+1)/2\rfloor$,  with $a_{1,1}= 1$, and $a_{n,k}=0$ for $k\leq 0$ or $k > \lfloor (n+1)/2\rfloor$.



$$
\begin{array}{c|cccc}
\hbox{$n$}\backslash\hbox{$k$}&1&2&3&4\\
\hline
1& 1&\\
2& 1&\\
3& 1&2&&\\
4& 1&8&&\\
5& 1&22&16&\\
6& 1&52&136&\\
\end{array}
\qquad\qquad\qquad
\begin{array}{c|cccc}
\hbox{$n$}\backslash\hbox{$k$}&0&1&2&3\\
\hline
1& 1&\\
2& 1&4\\
3& 1&20&&\\
4& 1&72&80&\\
5& 1&232&976&\\
6& 1&716&766&3904\\
\end{array}
$$
\centerline{Table 1. The first values of $(a_{n,k})$ and $(b_{n,k})$}
\smallskip

The classical Eulerian polynomials are the descent polynomials of the symmetric group or \emph{Coxeter group of type $A$}.
Analogues of   Eulerian polynomials for other Coxeter groups  were introduced and studied from a combinatorial point of view in the last three decades.
For instance,  the Eulerian polynomials $B_n(t)$  of type $B$ are defined by 
\begin{align}
\sum_{n\geq 0} (2k+1)^n t^n =\frac{B_n(t)}{(1-t)^{n+1}}.
\end{align}
The type $B$  version of \eqref{eq:FS} appeared
quite recently (see \cite{Pe07, St08, Chow08}) and  reads as follows
\begin{align}\label{eq:FSB} 
B_n(t)=\sum_{k=0}^{\lfloor n/2\rfloor}b_{n,k} t^k (1+t)^{n-2k},
\end{align}
where $b_{n,k}$ are  positive integers  satisfying the recurrence relation
\begin{align}
b_{n,k}=(2k+1)b_{n-1,k}+4(n+1-2k)b_{n-1,k-1},
\end{align}
for $n\geq 2$ and $0\leq k \leq \lfloor n/2\rfloor$,  with $b_{1,0}= 1$, and $b_{n,k}=0$ for $k\leq 0$ or $k > \lfloor n/2\rfloor$.

The  numbers  $a_{n,k}$ and $4^{-k}b_{n,k}$  appear as $A101280$  and $A008971$ in  The \emph{On-Line Encyclopedia of Integer Sequences}~:
 http://oeis.org.  
 
The aim of  this paper  is to prove a $q$-analogue of  \eqref{eq:FS}  with a refinement of  the triangle 
$(a_{n,k})$ for Carlitz's $q$-Eulerian polynomials~\cite{Ca75}, and also  a $q$-analogue of   \eqref{eq:FSB} with a refinement of  the triangle $(b_{n,k})$
for  Chow-Gessel's  $q$-Eulerian polynomials of type $B$ \cite{CG07}.  Note that some other extensions  of \eqref{eq:FS} are discussed in
 \cite{Br08,SW10,SZ10}.

\medskip

This paper is organized as follows: we derive in Section~2 a $q$-analogue of \eqref{eq:FS} using  Carlitz's $q$-Eulerian polynomials 
and   derive some results  about  the $q$-tangent number $T_{2n+1}(q)$  studied  in \cite{FH09}. In Section~3, we give a $q$-analogue of 
\eqref{eq:FSB} using  Chow-Gessel's $q$-Eulerian polynomials of type $B$, which yields new $q$-analogues of the secant numbers. 
 In Section~4, we apply our constructions to some
 conjectures  on the unimodality  from \cite{CG07}. Finally, we will briefly give some concluding remarks in the fifth and last section.

\section{A $q$-analogue for type $A$}
The $q$-binomial coefficient ${n\brack k }_q$ is defined by
$$
{n\brack k }_q=\frac{(q;q)_n}{(q;q)_k(q;q)_{n-k}},\qquad n\geq k\geq 0,
$$
where $(x;q)_n=(1-x)(1-xq)\cdots (1-xq^{n-1})$  and  $(x;q)_0=1$. 
Recall \cite{Ca54}  that 
Carlitz's  $q$-Eulerian polynomials $A_{n}(t,q):=\sum_{k=1}^{n-1}A_{n,k}(q)t^{k}$ can be defined 
by 
 \begin{align}\label{eq:carlitz}
 \sum_{k\geq 0} [k+1]_q^n t^k=\frac{A_{n}(t,q)}{(t;q)_{n+1}},
 \end{align}
 where $[n]_q=1+q+\cdots +q^{n-1}$.
 It is easy to see  that $A_{n,k}(q)$ satisfy the recurrence:
\begin{align}\label{eq: recurrence}
A_{n,k}(q)=[k]_{q}A_{n-1, k}(q)+q^{k-1}[n+1-k]_{q} A_{n-1, k-1}(q) \qquad  (1\leq k\leq n).
\end{align}

The following is our $q$-analogue of \eqref{eq:FS}.
\begin{thm}\label{th:A}
For any positive integer $n$, there are polynomials $a_{n,k}(q)\in\mathbb{N}[q]$ such that 
 the $q$-Eulerian polynomials  $ A_{n}(t,q)$ can be written as follows: 
 \begin{align}\label{eq:Chapo}
 A_{n}(t,q) =\sum_{k=1}^{\lfloor (n+1)/2\rfloor}a_{n,k}(q)t^{k-1}(-tq^{k};q)_{n+1-2k}.
 \end{align}
Moreover, the polynomials $a_{n,k}(q)$ satisfy the following recurrence relation
  \begin{align}\label{eq:Han}
 a_{n,k}(q)=[k]_{q}a_{n-1,k}(q)+(1+q^{k-1})q^{k-1}[n+2-2k]_{q}a_{n-1,k-1}(q)
\end{align}
for $n\geq 2$ and $1\leq k \leq \lfloor (n+1)/2\rfloor$, with
$a_{1,1}(q)=1$ and $a_{n,k}(q)=0$ 
for $k\leq 0$ or $k > \lfloor (n+1)/2\rfloor$. 
\end{thm}
\begin{proof}
 Assume that $a_{n,k}(q)$ are coefficients satisfying \eqref{eq:Han}. Then, by the $q$-binomial formula (cf.  \cite[Theorem 3.3]{An98}),
 \begin{align}\label{eq:qbinomial}
 (z;q)_{N}=\sum_{j=0}^{N}{N\brack j}_{q} (-z)^{j} q^{j(j-1)/2},
 \end{align}
 we see that \eqref{eq:Chapo} is equivalent to:
 \begin{align}\label{eq:qFSbis}
 A_{n,k}(q)=\sum_{s\geq 1} {n+1-2s\brack k-s}_{q} q^{(k-s)s+{k-s\choose 2}} a_{n,s}(q).
\end{align}
Substituting  \eqref{eq:qFSbis} in \eqref{eq: recurrence}, and using \eqref{eq:Han}, we derive:
\begin{align*}
&\sum_{s\geq 1}{n+1-2s\brack k-s}_{q}q^{(k-s)s+{k-s\choose 2}}\left([s]_{q}a_{n-1,s}(q)+(1+q^{s-1})q^{s-1}[n+2-2s]_{q}a_{n-1, s-1}(q)\right)\\
&=\sum_{s\geq 1}q^{(k-s)s+{k-s\choose 2}}\left([k]_{q} {n-2s\brack k-s}_{q}+
[n+1-k]_{q}{n-2s\brack k-1-s}_{q}\right)a_{n-1,s}(q).
\end{align*}
Extracting the coefficents of $a_{n-1,s}(q)$ we obtain:
\begin{align*}
{n+1-2s\brack k-s}_{q}[s]_{q} &+{n-1-2s\brack k-s-1}_{q}(1+q^{s})[n-2s]_{q}\\
&=[k]_{q} {n-2s\brack k-s}_{q}+ [n+1-k]_{q} {n-2s\brack k-1-s}_{q}.
\end{align*}
Canceling the common factors we get:
$$
[n+1-2s]_{q}[s]_{q}+[n-k-s+1]_{q}(1+q^{s})[k-s]_{q}=
[k]_{q}[n-k-s+1]_{q}+[n+1-k]_{q}[k-s]_{q}.
$$
The last identity is easy to verify, and this shows that \eqref{eq:Chapo} is satisfied.
 \end{proof}
The first values of the coefficients $a_{n,k}(q)$ read as follows:
$$\vcenter{\hbox{$\begin{array}{c|cccc}
\hbox{$n$}\backslash\hbox{$k$}&1&2&3&\\
\hline
 1&1&\\
 2&1&\\
 3&1&q+q^{2}&&\\
 4&1&2q(1+q)^{2}&&\\
 5&1&q(1+q)(3+5q+3q^2)&2q^3(1+q)^2(1+q^2)  &\\
 6& 1&q(1+q)^2 (4+5q+4q^2) &q^3(1+q)^2(1+q^2)(5+7q+5q^2) &\\
\end{array}$}
}
$$

\medskip

In \cite{FH09} Foata and Han defined a 
new  sequence of $q$-tangent numbers 
$T_{2n+1}(q)$  by
\begin{align}\label{eq:defT}
T_{2n+1}(q) = (-1)^n q^{\binom n2} A_{2n+1}(-q^{-n}, q).
\end{align}

We derive easily the following result from Theorem~\ref{th:A}, which is  the  most difficult part of the main result in \cite[Theorem 1.1]{FH09}.

\begin{coro}  The $q$-tangent number $T_{2n+1} (q)$ is a polynomial with positive integral coefficients. 
\end{coro}
\begin{proof} Let $a_{n,k}^*(q)=q^{-k(k+1)/2}a_{n,k}(q)$.  Then \eqref{eq:Han} becomes 
$$
a_{n,k}^*(q)=[k]_q a_{n-1,k}^*(q)+(1+q^{k-1}) [n+2-2k]_q a_{n-1,k-1}^*(q)
$$
with the same initial conditions as $a_{n,k}(q)$. This proves that $a_{n,k}^*(q)$ is a polynomial in $q$ with nonnegative integral  coefficients.
Now we show that $T_{2n+1} (q)= a_{2n+1, n+1}^*(q)$, which is sufficient to conclude.
Replacing $n$ by $2n+1$, $k$ by $n+1$, and $t$ by $-q^{-n}$ in (\ref{eq:Chapo}), we get
\begin{align*}
 A_{2n+1}(-q^{-n},q) 
	=\sum_{k=1}^{n+1}a_{2n+1,k}(q)(-q^{-n})^{k-1}(q^{k-n};q)_{2n+2-2k} 
	=a_{2n+1,n+1}(q)(-q^{-n})^{n}, 
 \end{align*}
since $(q^{k-n};q)_{2n+2-2k}=0$  for $k=1,2,\ldots, n$.  The result follows then from \eqref{eq:defT}.
\end{proof}

 We can also derive straightforwardly  the following result, which was proved in \cite{FH09} using combinatorics of the so-called \emph{doubloons}. 
\begin{coro}
The quotient $A_{2n}(t,q)/(1+tq^n)$  is a polynomial in $t$ and $q$  with 
{\it positive} integral coefficients.
\end{coro}
\begin{proof}
Note that 
$$
A_{2n}(t,q)=\sum_{k=1}^na_{2n,k}(q)t^{k-1}(-tq^k;q)_{2n+1-2k}.
$$
The result follows then from the fact that for $k=1,\ldots,n$, the coefficient
$(-tq^k;q)_{2n+1-2k}=(1+tq^k) \cdots (1+tq^{2n-k})$ contains the factor $1+tq^n$.
\end{proof}

For any  nonnegative  integer $n$,   set
\begin{align}\label{eq:f}
f_n(q):=\sum_{k=0}^{2n+1}{2n+1\choose k}\frac{(-1)^k}{1+q^{k-n}}.
\end{align}
Using the doubloon model, Foata-Han \cite{FH09} proved that 
$$
d_n(q):=\frac{T_{2n+1}(q)}{(1+q)(1+q^2)\dots(1+q^n)}=\frac{(-1)^{n+1}(-1;q)_{n+2}}{(1-q)^{2n+1}}f_n(q)$$
is a polynomial in $\N[q]$.  Actually  we can prove the integrality of   $d_n(q)$ without using the combinatorial device.

\begin{prop}\label{prop:T}
We have $d_n(q)\in\mathbb{Z}[q]$.
\end{prop}
\begin{proof}
Let $g_n(q)=(-1)^{n+1}(-1;q)_{n+2}$. Then $f_n(q)g_n(q)$ is clearly a polynomial in $\Z[q]$. 
 We must show that   1 is a  zero of order $2n+1$ of  the polynomial  $f_n(q)g_n(q)$ or
 $$d^p(f_n(q)g_n(q))/dq^n|_{q=1}=0\quad \textrm{for}\quad  p=0,\ldots, 2n. 
 $$
 By Leibniz's rule it suffices to show that 
 $f_n^{(p)}(1)=0$ for $p=0,\ldots, 2n$.

For any $k\in\mathbb{Z}$ and  $m\in \N$, 
we define 
the Laurent polynomial
 $P_{m,k}(x)$ by the relation:
$$ 
h_k^{(m)}(x)=\left(\frac{d}{dx}\right)^m(1+x^{k})^{-1}
=\frac{P_{m,k}(x)}{(1+x^{k})^{m+1}}.$$
Thus $P_{0,k}=1$, $P_{1,k}=-kx^{k-1}$, and for $m\geq0$,  we have
$$P_{m+1,k}(x)=(1+x^{k})P_{m,k}'(x)-k(m+1)x^{k-1}P_{m,k}(x).
$$
Therefore the $P_{m,k}$ can, for $m\geq1$, be written as follows:
$$P_{m,k}(x)=\sum_{l=1}^m\alpha_{l,m}x^{lk-m},$$
where $\alpha_{1,1}=-k$ and for $m\geq 1$, $\alpha_{1,m+1}=(k-m)\alpha_{1,m}$,
$\alpha_{m+1,m+1}=(m-k)\alpha_{m,m}$,
\begin{align*}
\alpha_{l,m+1}&=(lk-m)\alpha_{l,m}+(lk-mk-2k-m)\alpha_{l-1,m},\quad 2\leq l\leq m.
\end{align*}
This shows that for $m\geq 1$ and $1\leq l\leq m$,  the coefficient $\alpha_{l,m}$ is a polynomial in the variable $k$, with degree less than or equal to $m$. We deduce that $P_{m,k}(1)=\sum_{l=1}^m\alpha_{l,m}$ is also a polynomial in the variable $k$, with degree less than or equal to $m$, therefore we can write for some rational coefficients $a_j(m)$ only depending on $m$:
$$h_k^{(m)}(1)=\frac{P_{m,k}(1)}{2^{m+1}}=\sum_{j=0}^{m}a_j(m)k^j.$$
Thus, differentiating \eqref{eq:f} $m$ times  ($m\geq 0$) and then setting $q=1$, we get 
\begin{align*}
f_n^{(m)}(1)
&=\sum_{k=0}^{2n+1}{2n+1\choose k}(-1)^k\sum_{j=0}^{m}a_j(m)(k-n)^j\\
&=\sum_{j=0}^{2n}a_j(m) \sum_{k=0}^{2n+1}{2n+1\choose k}(-1)^k(k-n)^j.
\end{align*}
Now, applying $2n+1$ times  the finite difference operator $\Delta$  (defined by $\Delta f(x):=f(x+1)-f(x)$) to the polynomial $(n+1-x)^j$ ($0\leq j\leq 2n$)  and setting $x=0$
we get
$$
\left.\Delta^{2n+1}(n+1-x)^j\right|_{x=0}=\sum_{k=0}^{2n+1}{2n+1\choose k}(-1)^k(k-n)^j,
$$
which should  vanish because $(n+1-x)^j$ is a polynomial in $x$ of degree $j<2n+1$.
\end{proof}

\section{A $q$-analogue for type $B$}
A $B_n$-analogue of Carlitz's  $q$-Eulerian  polynomials are introduced by   Chow and Gessel~ \cite{CG07}. 
These  polynomials $B_n(t,q)$  are defined by
\begin{align}\label{eq:Beulerdef}
\sum_{k\geq 0} [2k+1]_q^n t^k=\frac{B_n(t,q)}{(t;q^2)_{n+1}}.
\end{align}
Let $B_n(t,q):=\sum_{k=0}^{n}B_{n,k}(q)t^{k}$. Then, the coefficients  $B_{n,k}(q)$ satisfy the recurrence relation~\cite[Prop. 3.2]{CG07}:
\begin{align}\label{eq:Beuler}
B_{n,k}(q)=[2k+1]_{q}B_{n-1, k}(q)+q^{2k-1}[2n-2k+1]_{q} B_{n-1, k-1}(q)\qquad 1\leq k\leq n.
\end{align}
We have the following $q$-analogue of \eqref{eq:FSB}.
\begin{thm}\label{th:B}
For any positive integer $n$, there are polynomials  $b_{n,k}(q)\in\mathbb{N}[q]$ such that
 the $q$-Eulerian polynomials of type $B$ can be written as follows: 
 \begin{align}\label{eq:qBFSbis}
 B_{n}(t,q)=\sum_{k=0}^{n}B_{n,k}(q)t^{k}=\sum_{k=0}^{\lfloor n/2\rfloor}b_{n,k}(q)t^{k}(-tq^{2k+1};q^{2})_{n-2k}.
 \end{align}
 Moreover, the coefficients $b_{n,k}(q)$ satisfy the following recurrence relation:
  \begin{align}\label{eq:BHan}
 b_{n,k}(q)=[2k+1]_{q}b_{n-1,k}(q)+(1+q)(1+q^{2k-1})q^{2k-1}[n+1-2k]_{q^{2}}b_{n-1,k-1}(q)
\end{align}
for $n\geq 2$ and $0\leq k \leq \lfloor n/2\rfloor$,  with $b_{1,0}(q)= 1$, and $b_{n,k}(q)=0$ for $k<0$ or $k > \lfloor n/2\rfloor$.
\end{thm}
\begin{proof}
 Assume that $b_{n,k}(q)$ are coefficients satisfying \eqref{eq:BHan}. Then, by applying  \eqref{eq:qbinomial} with the substitution $q\leftarrow q^2$,
 we derive that \eqref{eq:qBFSbis} is equivalent to:
 \begin{align}\label{eq:qB}
 B_{n,k}(q)=\sum_{s\geq 0} {n-2s\brack k-s}_{q^{2}} q^{k^{2}-s^{2}} b_{n,s}(q).
\end{align}
Substituting  \eqref{eq:qB} in \eqref{eq:Beuler}, and using \eqref{eq:BHan}, we get:
\begin{align*}
&\sum_{s\geq 0}{n-2s\brack k-s}_{q^{2}}q^{k^{2}-s^{2}}\left([2s+1]_{q}b_{n-1,s}(q)+(1+q)(1+q^{2s-1})q^{2s-1}[n+1-2s]_{q^{2}}b_{n-1, s-1}(q)\right)\\
&=\sum_{s\geq 0}q^{k^{2}-s^{2}}\left([2k+1]_{q} {n-1-2s\brack k-s}_{q^{2}}+
[2n+1-2k]_{q}{n-1-2s\brack k-1-s}_{q^{2}}\right)b_{n-1,s}(q).
\end{align*}
Extracting the coefficents of $b_{n-1,s}(q)$ we obtain:
\begin{align*}
{n-2s\brack k-s}_{q^{2}}[2s+1]_{q} &+{n-2-2s\brack k-s-1}_{q^{2}}(1+q)(1+q^{2s+1})[n-1-2s]_{q^{2}}\\
&=[2k+1]_{q} {n-1-2s\brack k-s}_{q^{2}}+ [2n+1-2k]_{q} {n-1-2s\brack k-1-s}_{q^{2}}.
\end{align*}
Canceling the common factors yields:
\begin{align*}
[n-2s]_{q^{2}} [2s+1]_{q}&+[n-k-s]_{q^{2}}(1+q)(1+q^{2s+1}) [k-s]_{q^{2}}\\
&=[2k+1]_{q}[n-k-s]_{q^{2}}+[2n+1-2k]_{q}[k-s]_{q^{2}}.
\end{align*}
The last identity is easy to verify, and this proves  \eqref{eq:qBFSbis}.
\end{proof}
For $n=1,\ldots, 4$, equation \eqref{eq:qBFSbis} reads:
\begin{align*}
B_{1}(t,q)&=1+qt;\\
B_{2}(t,q)&=(-tq;q^{2})_{2}+(q+2q^{2}+q^{3})t;\\
B_{3}(t,q)&=(-tq;q^{2})_{3}+(2q+5q^{2}+6q^{3}+5q^{4}+2q^{5}) t(1+tq^{3});\\
B_{4}(t,q)&=(-tq;q^{2})_{4}+(3q+9q^{2}+15q^{3}+18q^{4}+15q^{5}+9q^{6}+3q^{7})   t(-tq^{3};q^{2})_{2}\\
&\hspace{2.5cm}+(2q^{4}+7q^{5}+11 q^{6}+ 13q^{7}+ 14 q^{8}+13 q^{9}+11q^{10}+ 7q^{11}+2q^{12})t^{2}.
\end{align*}

\medskip

Theorem~\ref{th:B} implies immediately the following result, of which the first  was derived in \cite[Theorem 1.1 (d)]{FH10} with more work.
\begin{coro} For $n\geq 0$, we have 
\begin{align}
B_{2n+1}(-q^{-2n-1},q)&=0, \\
B_{2n}(-q^{-2n-1},q)&=(-1)^n q^{-n(2n+1)} b_{2n,n}(q). 
\end{align}
\end{coro}
\begin{proof}  By \eqref{eq:qBFSbis} we get
$$
B_{2n+1}(-q^{-2n-1},q)=\sum_{k=0}^nb_{2n+1,k}(q)(-q^{-2n-1})^k (q^{-2n+2k};q^2)_{2n+1-2k}=0.
$$
Substituting $n$ by $2n$ and $t$ by $-q^{-2n-1}$ in (\ref{eq:qBFSbis}) yields
 \begin{align*}
 B_{2n}(-q^{-2n-1},q)
	&=\sum_{k=0}^{n}b_{2n,k}(q)(-q^{-2n-1})^{k}(q^{-2n+2k};q^{2})_{2n-2k}\\
	&=(-1)^n q^{-n(2n+1)} b_{2n,n}(q).\qedhere
 \end{align*}
\end{proof}

The above result leads to define  a $q$-analogue of $B_{2n}(-1)=(-1)^n4^n E_{2n}$ (where the $E_{2n}$'s are the famous \emph{secant numbers}) by 
\begin{align}\label{eq:newsecant}
E_{2n}^*(q):=(-1)^n q^{n(n+1)} B_{2n}(-q^{-2n-1}, q).
\end{align}

\begin{thm} There is a polynomial $G_{2n}^*(q)\in\mathbb{Z}[q]$ such that $G^*_{2n}(1)=E_{2n}$ and 
$$
E^*_{2n} (q)= (1+q)(1+q^3)(1+q^5)\cdots (1+q^{2n-1})\cdot (1+q)^n \cdot 
G^*_{2n}(q).
$$
\end{thm}
\begin{proof}
Recall that $E_{2n}^*(q)=(-1)^nq^{n(n+1)} B_{2n}(-q^{-2n-1},q)$. 
From \eqref{eq:Beulerdef} we derive 
\begin{align*}
\frac{B_n(t,q)}{(t;q^2)_{n+1}}&=(1-q)^{-n}\sum_{j\geq 0}(1-q^{2j+1})^nt^j\\
&=(1-q)^{-n}\sum_{j\geq 0}t^j\sum_{k=0}^n {n\choose k} (-q^{2j+1})^k\\
&=(1-q)^{-n}\sum_{k=0}^n{n\choose k}\frac{(-q)^k}{1-tq^{2k}}.
\end{align*}
Substituting $n$ by $2n$ and setting $t=-q^{-2n-1}$ we obtain
$$
E_{2n}^*(q)
=(-1)^nq^{n(n+1)} \frac{(-q^{-2n-1};q^2)_{2n+1}}{(1-q)^{2n} } \sum_{k=0}^{2n}{2n\choose k} \frac{(-q)^k}{1+q^{2k-2n-1}}.
$$
Let
\begin{align*}
G_{2n}^*(q)&:=\frac{E_{2n}^*(q)}{(1+q)(1+q^3)\ldots (1+q^{2n-1})(1+q)^n}\\
&=(-1)^nq^{-n-1}
\frac{(-q;q^2)_{n+1}}{(1+q)^n(1-q)^{2n}}\sum_{k=0}^{2n}{2n\choose k} \frac{(-q)^k}{1+q^{2k-2n-1}}.
\end{align*}
For any  nonnegative  integer $n$,   set
\begin{align}\label{eq:f*}
f_n^*(q):=\sum_{k=0}^{2n}{2n\choose k}\frac{(-q)^k}{1+q^{2k-2n-1}}.
\end{align}
Let $g_n^*(q)=(-1)^{n}q^{-n-1} (-q;q^2)_{n+1}/(1+q)^n$. Then $f_n^*(q)g_n^*(q)$ is clearly a polynomial in $\Z[q]$. 
 We must show that   1 is a zero of order $2n$ of  the polynomial  $f_n^*(q)g_n^*(q)$ or
 $$d^p(f_n^*(q)g_n^*(q))/dq^n|_{q=1}=0\quad \textrm{for}\quad  p=0,\ldots, 2n-1. 
 $$
 By Leibniz's rule it suffices to show that 
 $d^p(f_n^*(q))/dq^p|_{q=1}=0$ for $p=0,\ldots, 2n-1$.
The rest of the proof is almost the same as that of Proposition~\ref{prop:T}, and is left to the reader.
\end{proof}

\begin{conj}  All the coefficients of the polynomials $G^*_{2n}(q)$  are positive.
\end{conj}

Since $G^*_{2n}(1)=E_{2n}$, the above conjecture would yield a new refinement  of the secant number.
\section{Application to unimodal problems}

A sequence $\{\alpha_0,\dots,\alpha_d\}$ is \emph{unimodal} if there exists an index $0\leq j\leq d$ such that $\alpha_i\leq\alpha_{i+1}$ for $i=1,\dots,j-1$ and $\alpha_i\geq\alpha_{i+1}$ for $i=j,\dots,d$. 
Chow and Gessel~\cite{CG07} studied   a kind of unimodality property of the $q$-Eulerian numbers assuming that $q$ is a real number. 
In this section,
we  derive some unimodal properties of  the sequences $(A_{n,k}(q))_{1\leq k\leq n}$ and $(B_{n,k}(q))_{1\leq k\leq n}$ from our previous results.
 From Theorem~\ref{th:A}, we are able to deduce the following corollary,  which provides a further support to Conjecture~4.8 in  \cite{CG07}.
\begin{prop}
Let $n\geq2$ be an integer and $j=\lfloor (n+1)/2\rfloor$. Then for $k=1,\dots,j-1$, we have $A_{n,k+1}(q)>A_{n,k}(q)$ if $q>1$ and $A_{n,n-k+1}(q)<A_{n,n-k}(q)$ if $q<1$.
\end{prop}
\begin{proof}
We start from \eqref{eq:qFSbis}, which can be rewritten
$$ A_{n,k}(q)=\sum_{s=1}^k {n+1-2s\brack k-s}_{q} q^{(k-s)(k+s-1)/2} a_{n,s}(q),$$
for $k=1,\dots,n$, where we assume $a_{n,s}(q)=0$ for $s>j$. Thus we can write for $k=1,\dots,j-1$:
\begin{multline*}
A_{n,k+1}(q)-A_{n,k}(q)=a_{n,k+1}(q)\\+\sum_{s=1}^k {n+1-2s\brack k+1-s}_{q} q^{(k+1-s)(k+s)/2} a_{n,s}(q)\left(1-q^{-k}\frac{1-q^{k+1-s}}{1-q^{n+1-k-s}}\right).
\end{multline*}
We know that the $q$-binomial coefficient is a polynomial in $q$ with nonnegative integer coefficients, and from Theorem~\ref{th:A} that this is also true for $a_{n,s}(q)$, $s=1,\dots,k+1$.
Therefore it is enough to show that the coefficient between brackets is nonegative for $1\leq s\leq k\leq j-1$. This coefficient can be rewritten as:
$$\frac{q^{n+1}-q^{k+s}+q^s-q^{k+1}}{q^{n+1}-q^{k+s}}.$$
Assume first that $q>1$. As $k+s\leq 2j-2\leq n-1<n+1$, the denominator of this fraction is positive. Moreover, it is not difficult to see that under the conditions $1\leq s\leq k\leq j-1$, and by using $(n-1)/2\leq j\leq (n+1)/2$, we have the following inequalities:
\begin{eqnarray*}
q^{n+1}-q^{k+s}+q^s-q^{k+1}&\geq& q^{n+1}-q^{2k}+q^k-q^{k+1}\\
&\geq&q^{n+1}-q^{2j-2}+q^{j-1}-q^{j}\\
&\geq& q^{n+1}-q^{n-1}+q^{(n-3)/2}-q^{(n+1)/2}.
\end{eqnarray*}
This last expression can be rewritten $(q^{(n+1)/2}-1)(q^{(n+1)/2}-q^{(n-3)/2})$ and is nonnegative, which shows that $A_{n,k+1}(q)\geq A_{n,k}(q)$ for $k=1,\dots,j-1$. \\

In the case $0<q<1$, we only need to use the well-known relation $A_{n,n-k+1}(q)=q^{n(n-1)/2}A_{n,k}(1/q)$ for any $k=1,\dots,n$, and the result is obvious from the case $q>1$. 
\end{proof}

\medskip
 
In the type $B$ case, it is conjectured in \cite[Conjecture~4.6]{CG07} that the sequence $(B_{n,k}(q))_{0\leq k\leq n}$ is unimodal. 
By Theorem~\ref{th:B}, we are able to confirm partially this conjecture.
\begin{prop}
Let $n\geq2$ be an integer and $j=\lfloor n/2\rfloor$. Then for $k=1,\dots,j-1$, we have $B_{n,k+1}(q)>B_{n,k}(q)$ if $q>1$ and $B_{n,n-k}(q)<B_{n,n-k-1}(q)$ if $q<1$.
\end{prop}
\begin{proof}
We start from \eqref{eq:qB}, which can be rewritten
$$ B_{n,k}(q)=\sum_{s=0}^k {n-2s\brack k-s}_{q^2} q^{k^2-s^2} b_{n,s}(q),$$
for $k=0,\dots,n$, where we assume $b_{n,s}(q)=0$ for $s>j$. Thus we can write for $k=0,\dots,j-1$:
\begin{multline*}
B_{n,k+1}(q)-B_{n,k}(q)=b_{n,k+1}(q)\\
+\sum_{s=0}^k {n-2s\brack k+1-s}_{q^2} q^{(k+1)^2-s^2} b_{n,s}(q)\left(1-q^{-2k-1}\frac{1-q^{2(k+1-s)}}{1-q^{2(n-k-s)}}\right).
\end{multline*}
We know that the $q$-binomial coefficient is a polynomial in $q$ with nonnegative integer coefficients, and from Theorem~\ref{th:B} that this is also true for $b_{n,s}(q)$, $s=0,\dots,k+1$.
Therefore it is enough to show that the coefficient between brackets is nonegative for $0\leq s\leq k\leq j-1$. This coefficient can be rewritten as:
$$\frac{q^{2n}-q^{2s+2k}+q^{2s-1}-q^{2k+1}}{q^{2n}-q^{2k+2s}}.$$
Assume first that $q>1$. As $k+s\leq 2j-2\leq n-2<n$, the denominator of this fraction is positive. Moreover, it is not difficult to see that under the conditions $0\leq s\leq k\leq j-1$, and by using $n/2-1\leq j\leq n/2$, we have the following inequalities:
\begin{eqnarray*}
q^{2n}-q^{2s+2k}+q^{2s-1}-q^{2k+1}&\geq& q^{2n}-q^{4k}+q^{2k-1}-q^{2k+1}\\
&\geq&q^{2n}-q^{4j-4}+q^{2j-3}-q^{2j-1}\\
&\geq& q^{2n}-q^{2n-4}+q^{n-5}-q^{n-1}.
\end{eqnarray*}
This last expression can be rewritten $(q^{2n}-q^{n-1})(1-q^{-4})$ and is nonnegative, which shows that $B_{n,k+1}(q)\geq B_{n,k}(q)$ for $k=0,\dots,j-1$. \\

In the case $0<q<1$, we only need to use the well-known relation $B_{n,n-k}(q)=q^{n^2}B_{n,k}(1/q)$ for any $k=0,\dots,n$, and the result is obvious from the case $q>1$. 
\end{proof}


\section{An  open problem on the combinatorial interpretations}

By  Theorems  1 and 5,   the polynomials  $a_{n,k}(q)$ and $b_{n,k}(q)$ have positive integral coefficients.  It is  then natural to ask the following question.
 \begin{Prob}\label{prob5}   What are the  combinatorial interpretations for $a_{n,k}(q)$  and $b_{n,k}(q)$?
 \end{Prob}

We  can give a  combinatorial interpretation for the \emph{odd central terms}
$a_{2n+1, n+1}(q)$ by using the {\it doubloon} model.
Recall  \cite{FH09} that a {\it doubloon} of order $(2n+1)$ is defined to be a permutation of the word $012\cdots (2n+1)$, represented as a
$2\times (n+1)$-matrix 
$\delta={a_0\,\cdots \,a_n\choose b_0\,\cdots\,b_n}$. 
Define 
$$
{\rm cmaj'}\, \delta:={\rm maj}(a_0\cdots a_nb_n\cdots b_0) 
- (n+1) {\rm des}(a_0\cdots a_nb_n\cdots b_0) +n^2,
$$ 
where ``des" and ``maj" are the usual {\it number of descents} and 
{\it major index} defined for words.
%
A doubloon $\delta={a_0\,\cdots \,a_n\choose
b_0\,\cdots\,b_n}$ is said to be {\it
interlaced}, if for every $k=1,2,\ldots,
n$ the sequence $(a_{k-1},a_k,b_{k-1},b_k)$ or one of its three
{\it cyclic rearrangements} is monotonic increasing or
decreasing. 
%
%
By Theorem~1.5 in \cite{FH09} we have the following result.
\begin{prop}\label{th:middle_a}
The polynomial $a_{2n+1, n+1}(q)$ is the generating function for the set of interlaced doubloons of order $2n+1$ by the statistic ${\rm cmaj}'$.
\end{prop}

Another  sequence of $q$-secant numbers is introduced in \cite{FH10} by 
$$
E_{2n}(q)=(-1)^n q^{n^2} B_{2n}(-q^{-2n}, q).
$$
Unfortunately, it seems not easy to relate our coefficients $b_{n,k}(q)$ from Section~3 to the doubloons of  type $B$, even for the central cases.

\section*{Acknowledgement}
This work was  partially supported by 
 the grant ANR-08-BLAN-0243-03.  The second author is grateful to Fr\'ed\'eric Chapoton for several  
  discussions at the initial stage  of this work.

\end{document}